\documentclass{amsart}

\usepackage{a4,amsmath,amssymb,amscd,verbatim,setspace}
\usepackage{mathtools}
\usepackage{marginnote}
\usepackage{wasysym}

\newtheorem{proposition}[equation]{Proposition}
\newtheorem{theorem}[equation]{Theorem}
\newtheorem{corollary}[equation]{Corollary}
\newtheorem{lemma}[equation]{Lemma}

\theoremstyle{definition}

\newtheorem{definition}[equation]{Definition}

\newtheorem{remark}[equation]{Remark}

\numberwithin{equation}{section}

\newcommand{\R} {\mathbb{R}}   \newcommand{\N} {\mathbb{N}}

\DeclareMathOperator{\vol}{vol}

\renewcommand{\emptyset}{1}

\begin{document}
\bibliographystyle{plain} \title[Orbit counting in conjugacy classes]
{Orbit counting in conjugacy classes for free groups acting on trees}

\author{George Kenison}
\address{Mathematics Institute, University of Warwick,
Coventry CV4 7AL, U.K.}
\email{G.Kenison@warwick.ac.uk}
\author{Richard Sharp} 
\address{Mathematics Institute, University of Warwick,
Coventry CV4 7AL, U.K.}
\email{R.J.Sharp@warwick.ac.uk}




\maketitle

%
%
%
%
%
%
%
%
%
\begin{abstract}
In this paper we study the action of the fundamental group of a finite metric graph on its universal covering tree.  We assume the graph is finite, connected and the degree of each vertex is at least three.  Further, we assume an irrationality condition on the edge lengths.  We obtain an asymptotic for the number of elements in a fixed conjugacy class for which the associated displacement of a  given base vertex in the universal covering tree is at most \(T\).
Under a mild extra assumption we also obtain a polynomial error term.
\end{abstract}

\maketitle

%
%
%
%
%
%
%
%
%

\section{Introduction}\label{in}
Let \(G\) be a finite connected graph. We always assume that each vertex of $G$ has degree
at least $3$, in which case the fundamental group of $G$ is a free group $F$ on $k \geq 2$ generators.
We make $G$ into a metric graph by assigning to each edge $e$ a positive real length $l(e)$.
The length of a path in \(G\) is given by the sum of the lengths of its edges.  
We assume the set of closed geodesics in \(G\) (i.e.\ closed paths without backtracking) has lengths not contained in a discrete subgroup of \(\mathbb R\).

The universal cover of \(G\) is an infinite tree \(\mathcal{T}\) and the metric on $G$ lifts to a metric $d_{\mathcal T}$ on $\mathcal T$.
We consider each edge in \(\mathcal{T}\) as an isometric copy of a real interval.  Then the ball of radius \(T\) centred at \(o\in \mathcal{T}\) is the set
	\begin{equation*}
		B(o,T) = \{y\in\mathcal{T} \colon d_\mathcal{T}(o,y) < T\}.
	\end{equation*}
The \textit{volume} of \(B(o,T)\) is the sum of the metric edge lengths in \(B(o,T)\).
  Let \(h>0\) denote the \textit{volume entropy} of \(\mathcal{T}\) given by
	\begin{equation*}
		\lim_{T\to\infty} \frac{1}{T} \log \vol( B(o,T)).
	\end{equation*}
We note the volume entropy is independent of the choice of \(o\in \mathcal{T}\).

Let \(x\in F\) and fix a base vertex \(o\in\mathcal{T}\).  We define \(L: F\to \mathbb R^+\) by \(L(x) = d_\mathcal{T}(o,ox)\).  
Guillop\'{e} \cite{guillope1994entropies} showed that
	\begin{equation*}
	 \#\{x\in F \colon  L(x)\leq T\} \sim c e^{hT} \quad \text{as } T\to\infty,
	\end{equation*}
for some \(c>0\). Here \(f(T)\sim g(T)\) means that \({f(T)}/{g(T)}\to 1\) as \(T\to\infty\).

Let 
$\mathfrak C$ be a non-trivial conjugacy class in $F$. Then $\mathfrak C$ is infinite and
it is interesting to study the restriction of the above counting problem to this conjugacy class, i.e.\ to study the asymptotic behaviour of
\begin{equation*}
 N_{\mathfrak C}(T) := \#\{ x\in \mathfrak C \colon L(x)\le T\}.
\end{equation*}
The following is our main result.
	\begin{theorem} \label{thm: mainresult} Suppose that \(G\) is a finite connected metric graph 
	such that the degree of each vertex is at least $3$ and the set of lengths of closed geodesics in \(G\) is not contained in a discrete subgroup of \(\mathbb R\).  Let $\mathfrak C$ be a non-trivial conjugacy class in $F$. Then, for some constant 
\(C>0\), depending on $\mathfrak C$,
	\begin{equation*}
		N_{\mathfrak C}(T) \sim C e^{hT/2}, \quad \text{as } T\to\infty.
	\end{equation*}
	\end{theorem}
	
We can also obtain a polynomial error term in our approximation to $N_\mathfrak{C}(T)$ subject
to a mild additional condition on the lengths of closed geodesics
(see Theorem \ref{thm: polyerror}).



%
In the case of a co-compact group of isometries of the hyperbolic plane, an analogue
of 
Theorem \ref{thm: mainresult}  was obtained by Huber \cite{Huber} in the 1960s.  If $\Gamma$ is
a co-compact Fuchsian group and $\mathfrak C$ is a non-trivial conjugacy class, he showed that
\[
\#\{g \in \mathfrak C \hbox{ : } d_{\mathbb H^2}(o,og) \leq T\} \sim C e^{T/2}, \quad
\hbox{as } T \to \infty,
\]
for some $C>0$ depending on \(\Gamma\) and \(\mathfrak{C}\) (while the unrestricted counting function
$\#\{g \in \Gamma \hbox{ : } d_{\mathbb H^2}(o,og) \leq T\}$ is asymptotic to a constant
times $e^T$).
Very recently, Parkkonen and Paulin \cite{parkkonen2013hyperbolic} have studied the same problem in higher dimensions 
and variable curvature, obtaining many results. In particular, they have shown that 
for the fundamental group of a compact negatively curved manifold acting on its universal cover,
the conjugacy counting function has exponential growth rate equal to $h/2$, where $h$ is
the topological entropy of the geodesic flow. They have an ergodic-geometric approach using,
in particular, the mixing properties of the Bowen--Margulis measure.

More closely related to our situation,
suppose that $G$ is a $(q+1)$-regular graph (i.e.\ each vertex has degree $q+1$)
with each edge given length $1$. Then Douma \cite{Douma} showed that
\[
N_{\mathfrak C}(n) \sim Cq^{\lfloor (n-l(\mathfrak C))/2\rfloor}, \quad \text{as } n\to\infty\ 
\]
(for $n \in \mathbb Z^+$), for some $C>0$, where $l(\mathfrak C)$ is the length of the closed geodesic in the conjugacy class $\mathfrak C$.

Since the first version of this paper was written, we learned that Broise-Alamichel, Parkkonen and Paulin also have results for graphs and metric graphs which include Theorem
\ref{thm: mainresult} and the result of Douma.
They also consider the more general situation of graphs of groups in the sense of Bass--Serre
theory. We understand that an 
account of this
work is in preparation \cite{BrPP}.

In contrast to the ergodic-geometric approach of Parkkonen and Paulin in \cite{parkkonen2013hyperbolic} or the use of spectral theory
of the graph Laplacian in
\cite{Douma} (which is inspired by Huber's original spectral approach \cite{Huber}), 
we use a method based on a symbolic coding of the group $F$ in terms of a subshift of finite type. We may then study a generating function via the spectra of a family of matrices. In the next section, we set out the background we shall need, discussing shifts of finite type and
some concepts from 
ergodic theory.  In section 3, we describe how
the lengths on the graph may be encoded in terms of a function on our subshift and use this function
to define a family of matrices and sketch a proof of
Guillop\'e's result given above. In section 4 we introduce a generating function appropriate to 
our problem and carry out an analysis which leads to the proof of Theorem \ref{thm: mainresult}.
In the final section we discuss error terms.

We are grateful to the referee for a very careful reading of our paper. Their comments have considerably
improved the exposition. 

%
%
%
%
\section{Preliminaries}
A (finite, connected) graph \(G=(V,E)\) consists of a finite collection of vertices \(V\) and edges \(E\).  Let \(E^o\) denote the oriented edge set of the graph \(G\).  For each \(e\in E^o\) we indicate the edge with reversed orientation by 
\(\overline{e}\). A \textit{path} is a sequence of consecutive oriented edges \(e_0, \ldots, e_{n-1}\); and we call a path \textit{non-backtracking} if, in addition, \(e_{i+1} \neq \overline{e_i}\) for \(i=0,\ldots, n-1\).  Path \(e_0,\ldots, e_{n-1}\) is said to be \textit{closed} if the terminal vertex of \(e_{n-1}\) and the initial vertex of \(e_0\) are the same.  We say a closed path \(e_0,\ldots, e_{n-1}\) is a \textit{closed geodesic} if the path is non-backtracking and \(e_{n-1}\neq \overline{e_0}\), i.e.\ each path given by a cyclic permutation of \(e_0,\ldots, e_{n-1}\) is non-backtracking.

The condition that each vertex has degree at least $3$ ensures that the fundamental
group of $G$ is a free group $F$ on $k \geq 2$ generators and that the universal cover is an infinite tree
$\mathcal T$. We put a metric on $G$ by assigning a positive length to each edge and this lifts to a metric on $\mathcal T$.

Fix a generating set  \(\mathcal{A} = \{a_1, \ldots, a_k\}\), \(k\ge 2\),
for the free group $F$ and write \(\mathcal{A}^{-1} = \{a_1^{-1}, \ldots, a_k^{-1}\}\).  We shall say \(x_0\cdots x_{n-1}\), with each \(x_i\in\mathcal{A}\cup\mathcal{A}^{-1}\), is a \textit{reduced word} 
if \(x_{i+1}\neq x_{i}^{-1}\) for \(i=0,\ldots ,n-1\) and a \textit{cylically reduced word} if, in addition, \(x_{n-1}\neq x_0^{-1}\).  Each non-identity element of \(F\) has a unique representation as a reduced word and the \textit{word length} \(|x|\) of \(x\) is the number of terms in its reduced word representation. 
We regard the identity element $1$ as a reduced word of length zero.
For each
$m \geq 0$, let \(W_m\) denote the set of reduced words of length at most $m$ and let
 \(W_m'\) denote the set of reduced words of length exactly $m$.  
 We let $W^* = \bigcup_{n=0}^\infty W_n'$ denote the set of all finite reduced words.

Let $\mathfrak C$ be a non-trivial conjugacy class in $F$. Then $\mathfrak C$ contains 
a cyclically reduced word $x_0 \cdots x_{n-1} \in W^*\setminus \{1\}$. The only other cyclically reduced words
in this conjugacy class are obtained by cyclic permutation and the elements of $\mathfrak C$ represented by non-cyclically reduced words have word length greater than $n$.

Recall the function  \(L: F\to \mathbb R^+\) given by \(L(x) = d_\mathcal{T}(o,ox)\). We will use the following lemma from \cite{sharp2010comparing}.

\begin{lemma}
\label{lem: ntoN} There exists \(N\in\mathbb N\) such that if \(n\ge N\) and \(x_0\cdots x_{n-1}\) is a reduced word then
	\[ L(x_0 x_1\cdots x_{n-1}) - L(x_1\cdots x_{n-1}) = L(x_0 x_1\cdots x_{N-1}) - L(x_1\cdots x_{N-1}).\]
\end{lemma}

It is useful to consider infinite sequences and dynamics on them. 
In particular, it is convenient to associate closed geodesics in the graph to 
periodic orbits of the shift map. We can define a function on both finite and infinite
reduced words which will encode the lengths $L(x)$ and which will also
give the lengths of closed geodesics by summing the function around the corresponding periodic orbits. Introducing the shift map on infinite sequences also has the advantage of allowing us to use
thermodynamic concepts from ergodic theory, for example pressure and equilibrium states 
defined below, and standard results about differentiating pressure.
Let $\Sigma$ denote the set of infinite reduced words in $\mathcal A \cup \mathcal A^{-1}$:
\[
\Sigma = \left\{(x_n)_{n=0}^\infty \hbox{ : } x_n \in \mathcal A \cup \mathcal A^{-1}
\ \mathrm{and} \ x_{n+1} \neq x_n^{-1} \ \forall n \in \mathbb Z^+\right\}.
\]
We will need to study the dynamical system on $\Sigma$ given by the shift map
$\sigma : \Sigma \to \Sigma$, defined by
$\sigma((x_n)_{n=0}^\infty) = (x_{n+1})_{n=0}^\infty$. This is a subshift of finite type and we will
refer to \cite{PP} for the general theory of these systems.

Writing \(x =(x_n)_{n=0}^\infty,y =(y_n)_{n=0}^\infty \in\Sigma\), 
we endow 
\(\Sigma\) with the metric defined by \(d(x,x)=0\) and,  for \(x\neq y\), \(d(x,y) = (1/2)^k\), where 
\(k = \min\{n\in\mathbb N \colon x_n\neq y_n\}\). This makes $\Sigma$ into a Cantor set
and makes the map $\sigma$ continuous. Furthermore, $\sigma$ is topologically mixing
(i.e. if $U,V \subset \Sigma$ are non-empty open sets then $\sigma^{-n}U \cap V \neq \varnothing$
for all sufficiently large $n$).

A function \(f:\Sigma \to\mathbb R\) is said to be \textit{locally constant} if there exists 
an \(N\in\mathbb N\) such that for any two elements \(x,y\in\Sigma\) with \(x_n=y_n\) for 
every \(0\le n\le N\) we have \(f(x)=f(y)\).  
If $f$ is locally constant 
then it is  
\textit{H\"{o}lder continuous} for every positive exponent (i.e.\ where 
$f$ is H\"older continuous  of exponent $\alpha>0$ if
there exists a positive constant \(\kappa\) such that
$|f(x) -f(y)| \le \kappa d(x,y)^\alpha$,
for every \(x,y\in\Sigma\)).  We will use the notation
\[
f^n = f + f\circ\sigma + \cdots + f\circ\sigma^{n-1}.
\]

Two H\"{o}lder continuous functions \(f,g:\Sigma
\to\mathbb R\) are said to be \textit{cohomologous} if \(f=g +u\circ \sigma -u\) for some continuous function 
\(u:\Sigma \to\mathbb R\).  
Clearly, two cohomologous functions have the same integral with respect to each
$\sigma$-invariant measure.

Let \(\mathcal{M}\) denote the set of \(\sigma\)-invariant probability measures on 
\(\Sigma\).  We denote by \(h(\mu)\) the measure theoretic entropy of $\sigma$
with respect to
\(\mu\).  For a continuous function \(f: \Sigma \to\mathbb R\), we define its \textit{pressure} 
\(P(f)\) by
	\begin{equation*}
	 P(f) = \sup_{\mu\in\mathcal{M}} \biggl( h(\mu) + \int f \, d\mu \biggr).
	\end{equation*}
We say that \(m\in\mathcal{M}\) is an \textit{equilibrium state} for \(f\) if the supremum is attained at \(m\).  When \(f\) is a H\"{o}lder continuous function the Variational Principle (\cite{PP}, Theorem 3.5) tells us the equilibrium state is unique.

\begin{remark}
Another subshift of finite type naturally associated to the graph is obtained by taking
infinite paths, i.e. infinite sequences of oriented edges with the restriction that
$e'$ can follow $e$ only if $e$ terminates at the initial vertex of $e'$.
The advantage of our approach is that is makes it easier to 
systematically enumerate the elements of a given conjugacy class. On the other hand, it requires more 
work to represent the edge lengths and we introduce a function that does this below.
\end{remark}

We will also need to consider the spectra of non-negative matrices and complex matrices.  
We say that a non-negative (square) matrix $A$ is {\it irreducible} if, for each 
pair of indices $(i,j)$, there exists 
$n \geq1$ such that $A^n(i,j)>0$ and that $A$ is {\it aperiodic} if there exists $n \geq 1$ such that, 
for each pair of indices $(i,j)$, $A^n(i,j)>0$.
We will use the following two standard results.

\begin{theorem}[Perron--Frobenius Theorem \cite{Gant}] \label{thm: Perron-Frobenius} Suppose that \(A\) is 
an aperiodic matrix with non-negative entries.  Then \(A\) has a simple and positive eigenvalue 
\(\beta\) such that \(\beta\) is strictly greater in modulus than all the remaining eigenvalues of \(A\).  
The left and right eigenvectors associated to the eigenvalue \(\beta\) have strictly positive entries.  
Moreover, \(\beta\) is the only eigenvalue of \(A\) that has an eigenvector whose entries are all 
non-negative.
\end{theorem}

%
\begin{theorem}[Wielandt's Theorem \cite{Gant}] \label{thm:Wielandt} Suppose that \(A\) is a square matrix with complex entries
and let \(|A|\) be the matrix whose entries are given by \(|A|(s,s') = |A(s,s')|\).  Suppose 
further that \(|A|\) is aperiodic and let $\beta$ be its maximal eigenvalue guaranteed by the 
Perron--Frobenius Theorem. Then the moduli of the eigenvalues of \(A\) are bounded above by \(\beta\).  Moreover, \(A\) has an eigenvalue of the form \(\beta e^{i \theta}\) (with \(\theta\in[0,2\pi]\)) if and only if  \(A= e^{i \theta}D|A|D^{-1}\) where \(D\) is a diagonal matrix whose entries along the main diagonal all have modulus one.  
\end{theorem}


\section{Length Functions, Matrices and Spectra}

We will prove Theorem \ref{thm: mainresult} by studying the analytic properties 
of a generating function $\sum_{x \in \mathfrak C} e^{-sL(x)}$. More precisely, we will show that the
generating function is analytic in the half-plane
$\mathrm{Re}(s) >h/2$, has a simple pole at $s=h/2$ and, crucially, apart from this pole has an analytic extension to a neighbourhood of $\mathrm{Re}(s) =h/2$. 
To do this, we will 
show that the generating function can essentially be written in terms of a family of
weighted matrices
and their eigenvectors. The required analytic properties will then follow from results 
about the spectra of these matrices. In turn, the key spectral property
(see Lemma \ref{lem: Wielandtlem} below) is a consequence of the hypothesis the lengths of 
closed geodesics in our metric graph do not lie in a discrete subgroup of the real numbers.

In this section, we will set up the machinery required to study the generating function.
We first introduce a function defined on (finite and infinite) reduced words which encodes 
information about lengths on the graph.
(This is similar to 
the constructions in
\cite{sharp2009distortion} and \cite{sharp2010comparing}.)
We will then introduce our weighted matrices and establish
some of their properties.

\begin{definition} \label{def: r}
We define a function $r : \Sigma \to \mathbb R$ by
\[
r((x_i)_{i=0}^\infty) = L(x_0 x_1\cdots x_{N-1}) - L(x_1\cdots x_{N-1}).
\]
We also define $r : W^* \to \mathbb R$ by $r(1)=0$ and, for $n \geq 1$,
\[
r(x_0 \cdots x_{n-1}) = L(x_0 \cdots x_{n-1}) -L(x_1 \cdots x_{n-1}).
\]
Note that, by Lemma \ref{lem: ntoN}, if $n \geq N$ then
\[
r(x_0 \cdots x_{n-1}) = L(x_0 \cdots x_{N-1}) -L(x_1 \cdots x_{N-1}).
\]
\end{definition}

We may also extend the definition of the shift map $\sigma$ to finite reduced words
by defining $\sigma : W^* \to W^*$ by 
\[\sigma(x_0 x_1 \cdots x_{n-1}) = (x_1 \cdots x_{n-1})\] 
and $\sigma 1=1$. We shall continue to write
$r^n = r + r \circ \sigma + \cdots + r \circ \sigma^{n-1}$.
The next lemma is immediate from the definition of $r$.

\begin{lemma}
\label{lem: Linr}
For any finite reduced word $x_0 \cdots x_{n-1}$, we have $L(x_0 \cdots x_{n-1}) = r^n(x_0 \cdots  x_{n-1})$. 
\end{lemma}

The following lemma connects closed geodesics in the graph \(G\) to periodic points in the subshift of finite type $\sigma : \Sigma \to \Sigma$. The fact that the sums of $r$ over periodic points
do not lie in a discrete subgroup will be crucial in establishing that our generating function has no non-real
poles on its abscissa of convergence.
\begin{lemma} \label{lem: geodesiclength}
Let $\gamma$ be the unique closed geodesic corresponding to the periodic orbit
$\{x,\sigma x,\ldots,\sigma^{n-1}x\}$ ($\sigma^n x=x$)  with \(x=(x_i)_{i=0}^\infty \in \Sigma\). Then
 \(r^n(x) = l(\gamma)\).
 In particular,
 \(\{r^n(x) \colon \sigma^nx=x, n \geq 1\}\) is not contained in a discrete 
 subgroup of $\mathbb R$.
\end{lemma}

\begin{proof}
Let $(x_0 \cdots x_{n-1})^{m}$ denote the $m$-fold concatenation of the
cyclically reduced word $x_0 \cdots x_{n-1}$.
By the definition of $r$ and Lemma \ref{lem: Linr}, we have
 \[
\left|r^{mn}(x) - L((x_0 \cdots x_{n-1})^m)\right| \leq 2N\|r\|_\infty
 \]
 and so, since $r^{mn}(x)=mr^n(x)$, 
 we deduce
 \[
 r^n(x) = \lim_{m \to \infty} \frac{1}{m} L((x_0 \cdots x_{n-1})^m).
 \]
 
 Now consider the closed geodesic $\gamma$ in $G$.
 This lifts to a geodesic path in $\mathcal T$, from some vertex $p$
 to $pg$, where $g \in F$ is conjugate to $x_0 \cdots x_{n-1}$.
 For each $m \geq 1$, we have $ml(\gamma) = d_{\mathcal T}(p,pg^m)$.
 For any vertex $q \in \mathcal T$, the triangle inequality gives
 \[
 d_{\mathcal T}(p,pg^m)- 2d_{\mathcal T}(p,q)
 \leq d_{\mathcal T}(q,qg^m) \leq d_{\mathcal T}(p,pg^m) + 2d_{\mathcal T}(p,q),
 \]
 so that
 \[
 l(\gamma) = \lim_{m \to \infty} \frac{1}{m} d_{\mathcal T}(q,qg^m).
 \]
 If $x_0 \cdots x_{n-1} = w^{-1}gw$ then putting $q = ow^{-1}$ gives
\begin{align*}
 l(\gamma)
 &= \lim_{m \to \infty} \frac{1}{m} d_{\mathcal T}(ow^{-1},ow^{-1}g^m)
 = \lim_{m \to \infty} \frac{1}{m} d_{\mathcal T}(o,o(x_0 \cdots x_{n-1})^m) \\
 &= \lim_{m \to \infty} \frac{1}{m} L((x_0 \cdots x_{n-1})^m).
 \end{align*}
 This completes the proof.
\end{proof}

Now we turn to the definition of the matrices we use to analyse the generating function.

We begin by defining an unweighted transition matrix \(A\), whose rows and columns are indexed by \(W_{N-1}\),
the set of reduced words of length at most $N-1$,
where the number
$N \geq 2$ is given by Lemma \ref{lem: ntoN}.  
The entries of $A$ are defined as follows.
We have $A(x,y) =1$ if there exists $n \leq N-2$ and there exists $x_0,x_1,\ldots, x_{n-2} \in \mathcal A \cup
\mathcal A^{-1}$ such that $x$ has the reduced word representation 
$x = x_0x_1 \cdots x_{n-2}$ and $y$ has the reduced word representation 
$y = x_1 \cdots x_{n-2}$, or if there exists $x_0,x_1,\cdots,x_{N-2},y_{N-2} \in \mathcal A
\cup \mathcal A^{-1}$ such that $x$ has the reduced word representation 
$x=x_0x_1 \cdots x_{N-2}$ and $y$ has the reduced word representation
$y = x_1 \cdots x_{N-2}y_{N-2}$. We have $A(x,y)=0$ in all other cases.

We next define the numbers we shall use to define weighted matrices compatible with $A$.

\begin{definition} \label{def: rformatrix}
For each pair $(x,y) \in W_{N-1} \times W_{N-1}$ with $A(x,y)=1$, we define a number $R(x,y)$ by
\[
R(x,y) = r(x_0 x_1 \cdots x_{N-2} y_{N-2})
= L(x_0x_1\cdots x_{N-2}y_{N-2}) - L(x_1 \cdots x_{N-2}y_{N-2}),
\]
in the case where $x= x_0 x_1 \cdots x_{N-2}$ and $y=x_1 x_2 \cdots x_{N-2} y_{N-2}$, 
and
\[
R(x,y) = r(x_0 x_1 \cdots x_{n-2})
= L(x_0x_1\cdots x_{n-2}) - L(x_1 \cdots x_{n-2}),
\]
in the case where $x= x_0 x_1 \cdots x_{n-2}$ and $y=x_1 x_2 \cdots x_{n-2}$,
with $n \leq N$.
\end{definition}

\medskip

We now introduce the family of weighted
matrices with which we encode edge lengths.  The matrices \(A_s\), with \(s\in \mathbb{C}\), 
have rows and columns indexed by \(W_{N-1}\) with entries
\begin{equation*}
A_s(x,y) = \begin{dcases} 0 & \text{if}\ A(x,y)=0, \\ 
e^{-sR(x,y)} & \text{otherwise}.
		\end{dcases}
\end{equation*}

We will be interested in the spectral properties of \(A_s\).  
If $s \in \mathbb R$ then $A_s$ has non-negative entries but it is not aperiodic or even
irreducible and we cannot apply the Perron--Frobenius Theorem directly. Similarly, we cannot apply Wielandt's Theorem directly to $A_s$ when $s \in \mathbb C$. 
Instead, we will consider the submatrix of \(A_s\) with rows and columns indexed by 
\(W_{N-1}'\), which we will denote by $B_s$. 
One can easily see that $A_s$ and $B_s$ have the same non-zero spectrum, though
$A_s$ has additional zero eigenvalues.

When \(s\in\R\), the matrix \(B_s\) is aperiodic and so, by the Perron--Frobenius Theorem, has a 
simple and positive eigenvalue \(\beta(s)\), which is strictly greater in modulus than all of the 
other eigenvalues of \(B_s\).  Let \(\tilde{u}(s)\) and \(\tilde{v}(s)\) denote the left and right 
eigenvectors of \(B_s\) 
corresponding to 
\(\beta(s)\),
normalised so that $\tilde u(s)$ is a probability vector and $\tilde u(s) \cdot \tilde v(s) =1$.

 We have the following lemma.
\begin{lemma} \label{lem: beta_decreasing}  For \(s\in\R\), suppose that \(\beta(s)\) is the maximal eigenvalue of the matrix \(B_s\).  Then $\beta(s)$ is real analytic and is related to the integral of 
\(r: \Sigma \to \R\) by the formula
 \begin{equation*} \beta'(s) = - \beta(s) \int r \, d\mu_{-sr}, \end{equation*}
 where $\mu_{-sr}$ is the equilibrium state for $-s r$.
 Furthermore, $\int r \, d\mu_{-s r} >0$.
\end{lemma}

\begin{proof}
The analyticity of an isolated simple eigenvalue is standard.
Furthermore, $\beta(s) = \exp P(-sr)$, where $P$ is the pressure function for 
the shift $\sigma : \Sigma \to \Sigma$.  The formula for the derivative is then given in \cite{PP},
for example.
For any $n \geq 1$, $r$ is cohomologous to $n^{-1}r^n$ and, for $n$ sufficiently large,
$r^n$ is strictly positive. These observations prove positivity of the integral.
\end{proof}


\begin{corollary} \label{lem: beta(h)2}
There exists a unique positive real number \(h\) such that \(\beta(h)=1\).
\end{corollary}
\begin{proof}
By comparing with the trace of $B_s^n$, one can easily check that $\beta(0)>1$ and that 
$\beta(s)<1$ for sufficiently large $s$. By Lemma \ref{lem: beta_decreasing}, 
$\beta(s)$ is strictly decreasing, so the required number $h$ exists and is unique.
\end{proof}


We will also need to consider $B_s$ for $s \in \mathbb C$. In particular,
we have the following result.

\begin{lemma} \label{lem: Wielandtlem}
For $t \neq 0$, the matrix \(B_{h+it}\) does not have 1 as an eigenvalue.
\end{lemma}

\begin{proof}
Lemma \ref{lem: geodesiclength} tells us \(\{r^n(x) \colon \sigma^nx=x, \ n \geq 1\}\) is not contained in a discrete subgroup of \(\R\). It then follows from \cite{parry1983analogue} that the result holds.
\end{proof}

For $s \in \mathbb R$, 
let \(u(s)\) and \(v(s)\) denote, respectively, the left and right 
eigenvectors of \(A_s\) associated to the eigenvalue \(\beta(s)\).  We choose to normalise 
\(u(s)\) so that each entry is strictly positive and scale the entries so that \((u(s))_{\emptyset}=1\). 
(Note that here the subscript $1$ denotes the entry index by the identity element
in $F$, i.e. the word of length zero.)
We normalise \(v(s)\) so that the entries \((v(s))_y\) for \(y\in W'_{N-1}\) are strictly positive, 
whilst all other entries are \(0\).  We then scale the entries of \(v(s)\) so that \(u(s)\cdot v(s)=1\).

We conclude this section by explaining why the number $h$, from Corollary \ref{lem: beta(h)2}, that gives $\beta(h)=1$ is equal to 
the volume entropy.
We will do this by sketching a proof of Guillop\'e's theorem.
This will also serve as a prelude to the analysis in the next section, where the arguments will 
be given in more detail.


To evaluate the asymptotic behaviour of \(\#\{x\in F\colon L(x)\le T\}\) we first establish analytic properties of the complex generating function
	\begin{equation*}
		\eta(s) = \sum_{n=1}^\infty \sum_{|x|=n} e^{-sL(x)}. 
	\end{equation*}
Letting \(\mathsf{1} = (1,\ldots, 1)\) and \(\mathsf{e}_{x}\) denote the standard unit vector associated to \(x\in W_{N-1}\), we can rewrite \(\eta(s)\) in terms of the matrices $A_s$:
	\begin{equation*}
		\eta(s) = \sum_{n=1}^\infty {\mathsf{1}} A_s^n \mathsf{e}_{\emptyset}.
	\end{equation*}
This converges absolutely for $\beta(\mathrm{Re}(s))<1$, i.e.\ for $\mathrm{Re}(s)>h$.


Since $\beta(s)$ is a simple eigenvalue, it varies analytically and we can show that 
\(\eta(s) = c_0/(s-h) + \phi_1(s)\) where \(\phi_1(s)\) is an analytic function in a neighbourhood of \(s = h\) and \(c_0 = -(v(s)\cdot \mathsf{1})/\beta'(h)>0\). Furthermore, using Lemma \ref{lem: Wielandtlem}, we may show
that \(\eta(s)\) has no further poles on the line \(\mathrm{Re}(s)=h\).


The generating function \(\eta(s)\) is related to the counting function
\(N(T) := \#\{x\in F \colon L(x)\le T\}\)
by the following Stieltjes integral:
	\begin{equation*} \label{eq: Stieltjes}
	 \eta(s) = \int_0^\infty e^{-sT} \, dN(T).
	\end{equation*}
This enables us to apply the Ikehara--Wiener Tauberian Theorem.
	\begin{theorem}[Ikehara--Wiener Tauberian Theorem
	(\cite{PP}, Theorem 6.7)] \label{thm: tauberian} Suppose that
	the function $\eta(s) = \int_0^\infty e^{-sT} \, dN(T)$ is analytic for \(\mathrm{Re}(s)> h\), has a simple pole at \(s=h\), and the function 
	\(\eta(s) - c_0/(s-h)\) has an analytic extension to a neighbourhood of 
	\(\mathrm{Re}(s)\ge h\).  Then
	 \begin{equation*}
	  N(T) \sim  \frac{c_0e^{hT}}{h}, \quad \mathrm{as } \ T \to \infty.
	 \end{equation*}
	\end{theorem}
As a consequence, we recover Guillop\'e's result that
	\begin{equation*}
	 \#\{x\in F \colon L(x) \le T\} \sim  ce^{hT} \quad \text{as } T\to\infty,
	\end{equation*}
for some $c>0$, 
and hence that the constant $h$ defined by $\beta(h)=1$ is the volume entropy.


\section{A Complex Generating Function}

Suppose that \(\mathfrak{C}\) is a non-trivial conjugacy class of the free group \(F\).  Let \(\mathfrak{C}_n\) denote the set of \(x\in\mathfrak{C}\) such that \(|x|=n\) and suppose that
	\begin{equation*}
		k = \min_{x\in \mathfrak{C}} |x|.
	\end{equation*}
  The set of \(g\in\mathfrak{C}_k\) are precisely those elements in \(\mathfrak{C}\) whose reduced word representations are cyclically reduced.  In fact, if \(g=g_1\cdots g_k \in\mathfrak{C}_k\) then all the remaining elements of \(\mathfrak{C}_k\) are given by cyclic permutations of the letters in \(g_1\cdots g_k\).

The reduced word representation of each \(x\in\mathfrak{C}\) takes the form
	\begin{equation*}
		x = w^{-1}gw = w_m^{-1} \cdots w_1^{-1} g_1 \cdots g_k w_1 \cdots w_m
	\end{equation*}
with \(g=g_1\cdots g_k \in \mathfrak{C}_k\); and \(w=w_1\cdots w_m\in W'_m\) subject to the 
restriction \(w_1 \neq g_1, g_k^{-1}\) in order that no pairwise cancellation occurs when 
concatenating \(w^{-1}\), \(g\) and \(w\). For a given \(g=g_1\cdots g_k\) we say that 
\(w\in W'_m(g)\) if \(w\in W'_m\) and \(w_1 \neq g_1, g_k^{-1}\). Clearly \(\mathfrak{C}_n\) is non-empty if and only if \(n = k + 2m\) for \(m\in\N\).

For \(x\in\mathfrak{C}\) with sufficiently long word length, the following lemma gives a useful decomposition of \(L(x)\).
	\begin{lemma}  Suppose that \(x\in\mathfrak{C}\) such that \(x=w^{-1}gw\) here \(g\in\mathfrak{C}_k\) and \(w\in W'_m(g)\) with \(m\ge N-1\).  Let \(y = w_1\cdots w_{N-1}\) then
		\begin{equation*} L(x) = L(w^{-1} g w) = L(y^{-1}g y) -2L(y) + 2L(w).
		\end{equation*}
	\end{lemma}

  In order to study the counting problem in the conjugacy class $\mathfrak C$, 
  we define a generating function
	\begin{equation*}
		\eta_\mathfrak{C}(s) = \sum_{x\in \mathfrak{C}} e^{-sL(x)}.
	\end{equation*}
We use the reduced word representation of elements in \(\mathfrak{C}\) to rewrite \(\eta_\mathfrak{C}(s)\) as follows:
	\begin{align*}
		\eta_\mathfrak{C}(s) &= \sum_{n=1}^\infty \sum_{x\in \mathfrak{C}_n} e^{-sL(x)} \\
		&= \sum_{g\in\mathfrak{C}_k} \sum_{y \in W'_{N-1}(g)} e^{-s(L(y^{-1}gy) - 2L(y))} \sum_{m= N-1}^\infty \sum_{\substack{w\in W'_m \\ y = w_1\cdots w_{N-1}}} e^{-2sL(w)} + \phi(s)
	\end{align*}
where \(\phi(s)\) is an entire function. We may write \(\eta_\mathfrak{C}(s)\) in terms of \(A_{2s}\) by using
	\begin{equation*}
	\sum_{m=N-1}^\infty \sum_{\substack{w\in W'_m \\ y=w_1\cdots w_{N-1} }} e^{-2sL(w)} = \sum_{m=N-1}^\infty \mathsf{e}_y A_{2s}^m \mathsf{e}_\emptyset.
	\end{equation*}

We will need the following classical result from linear algebra.

	\begin{lemma} \label{lem: c=u.v}  Let \(M\) be a \(d\times d\) matrix with real entries.  Suppose that \(M\) has a simple eigenvalue \(\beta\), and that $u$ and \(v\) 
	are the associated left and right eigenvectors, normalised so that $u \cdot v=1$.  
	Then $w \in \mathbb R^d$ can be written
	$w = (u \cdot w)v + \overline{v}$, where $\overline{v}$ is in the span 
	of generalised right eigenvectors of \(M\) not associated to \(\beta\).	
	\end{lemma}
	
\begin{proof}
Let $\{v\} \cup \mathcal S$ be a Jordan basis for $M$. Then
$\mathbb R^d = \mathbb R v \oplus \mathrm{span}(\mathcal S)$ and $u$ is
orthogonal to each element of $\mathcal S$. The result follows.
\end{proof}

	 	\begin{proposition} \label{prop: countingetaanalytic}
 	 The generating function \(\eta_{\mathfrak C}(s)\) is analytic for $\mathrm{Re}(s)>h/2$, has a simple pole at \(s=h/2\) with positive residue and, apart from this, has an analytic extension
to a neighbourhood of $\mathrm{Re}(s) \geq h/2$. 
 	\end{proposition}

	\begin{proof}
For $\sigma \in \mathbb R$, $A_{2\sigma}$ has spectral radius $\beta(2\sigma)$. 
Since this is strictly decreasing and $\beta(h)=1$, 
it is clear that 
$\sum_{m={N-1}}^\infty \mathsf{e}_y^\text{T} A_{2\sigma}^m \mathsf{e}_{\emptyset}$
converges for $\sigma > h/2$ and hence that, for $s \in \mathbb C$, $\eta_{\mathfrak C}(s)$
is analytic for $\mathrm{Re}(s)>h/2$.


We now consider the analyticity of \(\eta(s)\) for \(s\) in a neighbourhood of 
\(h/2 +it \in \mathbb{C}\), for an arbitrary $t \in \mathbb R$. 
We will let $\mathrm{spr}(M)$ denote the spectral radius of a matrix $M$. 
By Wielandt's Theorem (Theorem \ref{thm:Wielandt}), either 
\begin{enumerate}
\item
$\mathrm{spr}(A_{h+2it}) = \mathrm{spr}(B_h)=1$, in which case 
$A_{h+2it}$ has a simple
eigenvalue $\beta(h+2it)$ with $|\beta(h+2it)|=\beta(h)=1$ and such that the 
remaining eigenvalues are strictly smaller in modulus; 
or 
\item
$\mathrm{spr}(A_{h+2it}) < \mathrm{spr}(B_h)=1$. 
\end{enumerate}
When (1) holds, standard eigenvalue perturbation theory gives that the simple eigenvalue
$\beta(2s)$ persists and is analytic for $s$ in a neighbourhood of $h+it$, as are the corresponding 
left and right eigenvectors ${u}(2s)$ and ${v}(2s)$ \cite{Kato}.
By Lemma \ref{lem: c=u.v} and recalling that $({u}(2s))_{\emptyset}=1$, we have
$
\mathsf{e}_{\emptyset} = {v}(2s) + \overline{v}(2s)
$,
where \(\overline{v}(2s)\) is a vector in the subspace spanned by the generalised eigenvectors associated to the non-maximal eigenvalues of $A_{2s}$. 
Thus 
we have
	\begin{align*}
	 \sum_{m={N-1}}^\infty \mathsf{e}_y^\text{T} A_{2s}^m \mathsf{e}_{\emptyset} 
	 &= \sum_{m=N-1}^\infty (\mathsf{e}_y \cdot {v}(2s)) \beta(2s)^m + \sum_{m=N-1}^\infty \mathsf{e}_y^\text{T} A_{2s}^m\overline{v}(2s)  \\
	 &= \frac{(\mathsf{e}_y \cdot {v}(2s)) \beta(2s)^{N-1}}{1 - \beta(2s)}
	 +\phi_2(s),
	\end{align*}
where $\phi_2(s)$ is analytic in a neighbourhood of $h/2+it$.
Therefore, $\eta_{\mathfrak C}(s)$ is analytic in a neighbourhood of $h/2+it$ unless
$\beta(h+2it)=1$. 
Lemma  \ref{lem: Wielandtlem} tells us that this only occurs when 
$t=0$.
When (2) holds, we immediately obtain that 
$\sum_{m={N-1}}^\infty \mathsf{e}_v^\text{T} A_{2s}^m \mathsf{e}_{\emptyset}$ converges and 
hence that $\eta_{\mathfrak C}(s)$ converges to an analytic function for $s$ in a neighbourhood of 
$h/2+it$. 

For $s$ in a neighbourhood of $h/2$, we have that, modulo an analytic function,
\begin{equation*}
\eta_{\mathfrak C}(s) = \sum_{g\in \mathfrak C_k} \sum_{y\in W'_{N-1}(g)} 
		e^{-s(L(y^{-1}gy)-2L(y))} \frac{(\mathsf{e}_y \cdot {v}(2s)) \beta(2s)^{N-1}}{1 - \beta(2s)}.
\end{equation*}
From the analyticity of $\beta$ and the fact that $\beta(h)=1$, we obtain
that, in a neighbourhood of $h/2$,
\[
\eta_{\mathfrak C}(s) = \frac{c}{s-h/2} + \phi_3(s),
\]
where $\phi_3(s)$ is analytic and $c>0$. The latter holds because
${\mathsf e}_y \cdot {v}(h) >0$, for each $y \in \bigcup_{g \in \mathfrak C_k} W'_{N-1}(g)$, and, by Lemma 
\ref{lem: beta_decreasing}, $-1/(2\beta'(h))>0$.

Combining the above observations, we have that $\eta_{\mathfrak C}(s)$ is analytic
for $\mathrm{Re}(s)>h/2$ and, apart from a simple pole at $s=h/2$, has an analytic extension to a 
neighbourhood of $\mathrm{Re}(s) \geq h/2$. Furthermore, the residue at the simple pole is positive.
\end{proof}

Since \(\eta_{\mathfrak C}(s)\) is related to the counting function \(N_{\mathfrak C}(T) = \#\{ x\in \mathfrak C \colon d_{\mathcal{T}}(o,ox) \le T\} \) via the Stieltjes integral
	\begin{equation*}
		\eta_{\mathfrak C}(s) = \int_0^\infty e^{-sT} \, dN_{\mathfrak C}(T),
	\end{equation*}
 we may apply the Ikehara--Wiener Tauberian Theorem (Theorem \ref{thm: tauberian}) to conclude that
	\begin{equation*}
		N_{\mathfrak C}(T) \sim C\frac{e^{h T/2}}{h/2},
	\end{equation*}
for some $C>0$. This completes the proof of Theorem \ref{thm: mainresult}.

%
%

%
\section{Error terms}

In this final section we discuss the error terms which may appear when estimating 
$N_{\mathfrak C}(T)$. We first note the following, which may be deduced from the analysis above and 
the arguments in \cite{Pol} (Propositions 6 and 7).

\begin{proposition}
There is never an exponential error term in Theorem \ref{thm: mainresult}, i.e.\ for no 
$\epsilon>0$ do we have
$N_{\mathfrak C}(T) = Ce^{hT/2} + O(e^{(h-\epsilon)T/2})$.
\end{proposition}

A more interesting problem is to ask when there is a polynomial error term for $N_{\mathfrak C}(T)$.
Recall that an irrational number $\alpha$ is said to be {\it Diophantine} if there 
exist $c>0$ and $\beta >1$ such that
$|q\alpha-p| \geq cq^{-\beta}$ for all $p,q \in \mathbb Z$, $q>0$.
We have the following.

\begin{theorem} \label{thm: polyerror}
Suppose that \(G\) is a finite connected metric graph 
such that the degree of each vertex is at least $3$.
Suppose also that  $G$ contains two closed geodesics 
$\gamma$ and $\gamma'$ such that
$l(\gamma)/l(\gamma')$ is Diophantine.
Then there exists $\delta>0$ such that
\begin{align*}
N_{\mathfrak C}(T) = Ce^{hT/2} + O(e^{hT/2} T^{-\delta}).
\end{align*}
\end{theorem}

Note that, since the lengths of the closed geodesics in $G$ are not contained in a discrete 
subgroup of $\mathbb R$ if and only if there are two closed geodesics the ratio of whose lengths
is irrational, the hypothesis of Theorem \ref{thm: polyerror} is strictly
stronger that that of Theorem \ref{thm: mainresult}.

The crucial new ingredient is to obtain a bound on powers of the matrix $A_s$, for $\mathrm{Re}(s)$
close to $h$ and $\mathrm{Im}(s)$ away from zero.
To do this, we use the work of Dolgopyat  \cite{dolgopyat1998prevalence} on transfer operators,
where it is a key ingredient
in studying the mixing rate of flows
(cf.\ also \cite{pollicottsharp2001error}).
Since $r : \Sigma \to \mathbb R$ is locally constant, the transfer operator $L_{-sr}$, defined pointwise
by
\[
(L_{-sr}\psi)(x) = \sum_{\sigma y=x} e^{-sr(y)} \psi(y),
\]
 acts on
the space of complex valued H\"older continuous functions with exponent $\alpha$, for any $\alpha>0$. 
For definiteness, we will consider the action on the Banach space of Lipschitz functions,
$L_{-sr} :  C_{\mathrm{Lip}}(\Sigma,\mathbb C) \to C_{\mathrm{Lip}}(\Sigma,\mathbb C)$,
with the norm $\|f\| = \|f\|_\infty +|f|_{\mathrm{Lip}}$, where
\[
|f|_{\mathrm{Lip}}= \sup_{x \neq y} \frac{|f(x)-f(y)|}{d(x,y)}.
\]
Let $V_{N-1} \subset C_{\mathrm{Lip}}(\Sigma,\mathbb C)$ denote the finite dimensional
subspace consisting of locally constant functions depending on the first $N-1$ coordinates.
This has dimension $D := \# W_{N-1}'$.
Then $L_{-sr} : V_{N-1} \to V_{N-1}$ and $\|L_{-sr}|_{V_{N-1}}\| \leq \|L_{-sr}\|$.
Furthermore, the restriction of $L_{-sr}$ to $V_{N-1}$ can be identified with the action of $B_{s}$
on $\mathbb C^D$. We can therefore use the 
results of \cite{dolgopyat1998prevalence} on bounding the norm
of iterates of transfer operators to give a bound on the norm of the 
iterates of $B_s$ and hence $A_s$.

For the remainder of this section, it will be convenient to write $s=\varsigma+it$.

\begin{proposition} \label{prop: Diophantine1}
Under the hypotheses of Theorem \ref{thm: polyerror}, there exist constants
 \(C_1,C_2>0\), $\tau>0$ and \(t_1 \ge 1\) such that when \(|t|\ge t_1\) and \(m\ge 1\),
	\begin{equation*}
		\| A_{\varsigma+it}^{2\nu m} \|\le C_1|t|\beta(\varsigma)^{2\nu m} 
		\left(1-|t|^{-\tau}\right)^{m-1},
	\end{equation*}
where \(\nu =\lfloor C_2\log|t| \rfloor\).
\end{proposition} 

We will also use the elementary inequality (cf.\ Proposition 2.1 of \cite{PP}): 

\begin{lemma} \label{lem: basicinequality}
There exists a constant $D_1>0$, independent of $t$ and uniform in $\varsigma$,
such that, for all $n >0$, we have
\begin{align*}
\|A_{\varsigma+it}^n\| \leq  \beta(\varsigma)^n \left(D_1|t| + 2^{-n}\right).
\end{align*}
\end{lemma}

\begin{proof}
Let $\phi$ denote the strictly positive eigenfunction of the operator $L_{-\varsigma r}$ associated to
the eigenvalue $\beta(\varsigma)$, guaranteed by the Ruelle-Perron-Frobenius Theorem
for this operator (see \cite{PP}, Theorem 2.2). Let 
\[
g = -\varsigma r + \log \phi - \log \phi \circ \sigma -\log \beta(\varsigma).
\]
Then $L_g1=1$ and
it follows from Proposition 2.1 of \cite{PP} that
\[
|L_{g-itr}^n\psi|_{\mathrm{Lip}} \leq D_0|t| \|\psi\|_\infty + 2^{-n}|\psi|_{\mathrm{Lip}},
\]
for some $D_0>0$ independent of $t$ and uniform in $\varsigma$, and 
\[
\|L_{g-itr}^n\psi\|_\infty \leq \|\psi\|_\infty.
\]
(The only difference from the statement given in \cite{PP} is the appearance of the term $|t|$ 
and the uniformity of $D_0$ but this 
follows from an inspection of the proof.)
In particular, we have
\[
 \|L_{g-itr}^n\psi\| \leq (D_1|t| + 2^{-n})\|\psi\|.
 \]
 for some $D_1>0$ independent of $t$ and uniform in $\sigma$. Since
 \[
 L_{-(\varsigma+it)r} = \beta(\varsigma) \Delta_\phi L_{g-itr} \Delta_\phi^{-1},
 \]
 where $\Delta_\phi$ is the multiplication operator $\Delta_\phi(\psi)=\phi\psi$,
 the result follows.
\end{proof}

Combining Proposition \ref{prop: Diophantine1} and Lemma \ref{lem: basicinequality}
gives the following bound on \(\|A_s^n\|\) for all \(n\in\mathbb N\). 

\begin{proposition} \label{prop: allnbound}
Let \(n=2\nu m +l\) where \(m=\left\lfloor \tfrac{n}{2\nu} \right\rfloor\) and \(l\in\{0,\ldots, 2\nu -1\}\), then,
 for $|t|\geq t_1$
	\begin{equation*} 
	 \|A_{\varsigma+it}^n\| \le 
	 C_3|t|^2 \beta(\varsigma)^{n}(1- |t|^{-\tau})^{m-1},
	\end{equation*}
for some $C_3>0$ independent of $t$ and uniform in $\varsigma$.
\end{proposition}

We use this to study the analyticity of \(\eta_{\mathfrak{C}}(s)\) (using a simpler version of the
arguments in \cite{pollicottsharp2001error}).
\begin{proposition} 
There exist constants \(\rho>0\) and \(t_2 \ge t_1\) such that 
\(\eta_\mathfrak{C}(s)\) has an analytic extension to 
the region
 	\begin{equation*}
	 \mathcal{R}(\rho) = \{ \varsigma+it \in\mathbb C \colon 2\varsigma> h - 1/|2t|^\rho, 
	 |t|\ge t_2 \},
	\end{equation*}
where it satisfies the bound \(|\eta_\mathfrak{C}(\varsigma+it)|= O(|t|^{2+\rho})\).
\end{proposition}
 \begin{proof} 
 Suppose that \(\varsigma+it\in\mathbb C\) satisfies \(2\varsigma> h - |2t|^{-\rho}\) and 
 \(|2t|\ge t_1\).  Take $\rho >\tau$ (with \(\tau\) the constant from Proposition \ref{prop: Diophantine1}). Since
	\begin{equation*}
	  \beta(2\varsigma) = 1 + \beta'(h)(2\varsigma - h) 
	  + O((2\varsigma-h)^2)
	\end{equation*}
with \(\beta'(h)<0\), for \(|t|\ge t_2\), 
where \(t_2\ge t_1\), we have
$( 1 -|2 t|^{-\tau} )^{1/2\nu} < \beta(2\varsigma)^{-1}$.

We can estimate
\(|\eta_\mathfrak{C}(s)| \le M \sum_{m=1}^\infty | A_{2s}^m 
\mathsf{e}_\emptyset | \),
for some $M>0$. Thus, for \(\varsigma+it\in\mathcal{R}(\rho)\), we have, by Proposition \ref{prop: allnbound},
	\begin{align*}
	 \sum_{m=1}^\infty | A_{2(\varsigma+it)}^m \mathsf{e}_{\emptyset} | 
	 &\le \sum_{m=1}^\infty C_3 |2t|^2 \beta(2\varsigma)^m
	 \left( 1 - |2t|^{-\tau} \right)^{\left\lfloor m/2\nu  \right\rfloor -1} \\
		&\le \frac{C_3 |2t|^2 \beta(2\varsigma) }{ \left( 1 - |2t|^{-\tau} \right)^{2-2/\nu} 
		\bigl( 1- \beta(2\varsigma) \left( 1 - |2t|^{-\tau} \right)^{1/2\nu} \bigr)  } \\
		&\le \frac{C_4 |2t|^{2}}{1- \beta(2\varsigma) \left( 1- |2t|^{-\tau} \right)^{1/2\nu} } =O(|t|^{2+\rho}),
	\end{align*}
which shows that $\eta_{\mathfrak C}(s)$ is analytic in the desired region and gives the 
bound.
 \end{proof}

Let \(\xi_\mathfrak{C}(s)\) be the normalised generating function given by
	\begin{equation*}
	 \xi_\mathfrak{C}(s) = \eta_{\mathfrak{C}}(sh/2) = \sum_{x\in\mathfrak{C}} e^{-shL(x)/2}.
	\end{equation*}
	
We immediately deduce that
 \(\xi_\mathfrak{C}(s)\) is analytic in the half-plane \(\mathrm{Re}(s)>1\),
 has an analytic extension to a neighbourhood of
	  \(\mathrm{Re}(s)\ge 1\) apart from the simple pole at \(s=1\), which has positive residue,
	  and, furthermore, that
there exist positive constants \(\rho\) and \(t_3=2t_2/h\) such that \(\xi_\mathfrak{C}(s)\) 
has an analytic extension to
		\begin{equation*}
		 \mathcal{R}_\xi(\rho) = \left\{ \varsigma+it \colon \varsigma
		 > 1 - \frac{1}{h^{\rho+1}|t|^\rho},\ |t|> t_3\right\};\ \text{and}
		\end{equation*}
		where it satisfies
	\(|\xi_\mathfrak{C}(s)| = O(|t|^{2+\rho})\).


Let us introduce a normalised counting function 
\[
\psi_0(T) = \sum_{e^{hL(x)/2}\le T} 1.
\]    
Adapting the arguments of \cite{pollicottsharp2001error},
we will establish an error term for \(\psi_0(T)\), 
from which Theorem \ref{thm: polyerror} will follow since
$\psi_0(e^{hT/2}) = N_{\mathfrak C}(T)$.
We introduce the following family of auxiliary functions.  
Let \(\psi_1(T) = \int_1^T \psi_0(u) \, d{u}\) and continue inductively so that
	\begin{equation*}
	 \psi_k(T) = \int_1^T \psi_{k-1}(u) \, d{u} = \frac{1}{k!} \sum_{e^{hL(x)/2}\le T} (T- e^{hL(x)/2})^k.
	\end{equation*}

We use the following identity (\cite{ingham1932distribution}, Theorem B, page 31) to connect the functions \(\xi_\mathfrak{C}(s)\) and \(\psi_k(T)\).  If \(k\) is a positive integer and \(d>0\) we have
	\begin{equation*}
	 \frac{1}{2\pi i } \int_{d-i  \infty}^{d+i  \infty} \frac{y^s}{s(s+1)\cdots (s+k)} d{s} = \begin{cases} 0 & \text{if } 0<y<1, \\ \tfrac{1}{k!} (1-1/y)^k & \text{if } y\ge 1. \end{cases}
	\end{equation*}
This gives us the following.

\begin{lemma}
	For \(d>1\) we may write
		\begin{equation*}
		 \psi_k(T) = \frac{1}{2\pi i } \int_{d-i  \infty}^{d+i \infty}  
		 \frac{\xi_\mathfrak{C}(s) T^{s+k}}{s(s+1)\cdots (s+k)} d{s}.
		\end{equation*}
\end{lemma}

We briefly outline the method to approximate the integral for \(\psi_k(T)\). First, compare the 
integral for \(\psi_k(T)\) to the truncated integral on the line segment \([d-i  R, d+ i  R]\), 
where \(R=(\log T)^\varepsilon\) and \(0< \varepsilon< 1/\rho\).  Let us choose \(d=1+1/\log T\) 
and then since \(\xi_\mathfrak{C}(d) = O((d-1)^{-1})\), we deduce
	\begin{align*}
		\left|\psi_k(T) - \frac{1}{2\pi i } \int_{d-i  R}^{d+i  R} 
		 \frac{\xi_\mathfrak{C}(s) T^{s+k}}{s(s+1)\cdots (s+k)} d{s} \right| 
		= O \left( \frac{(\log T) T^{k+1}}{R^{k+1}} \right) 
		= O\left( \frac{T}{(\log T)^{k\varepsilon}} \right).
	\end{align*}

We evaluate the truncated integral using Cauchy's Residue Theorem.  Consider a closed contour 
\(\Gamma \cup [d-i  R, d+i  R]\). Here $\Gamma$ is the union of the 
line segments \([d+i  R, c(R) + i  R]\), \([c(R)-i  R, d-i  R]\) and \([c(R)+i  R, c(R)-i  R]\),
where 
\[c(R) = 1 - \frac{1}{2h^{\rho+1}R^\rho},
\] 
so that \(\Gamma\) lies 
in \(\mathcal{R}_\xi(\rho)\).  
Note that \(\Gamma \cup [d-i  R, d+i  R]\) encloses the simple pole of \(\xi_\mathfrak{C}(s)\) 
at $s=1$. Cauchy's Residue Theorem gives that, for some $c>0$,
	\begin{equation*}
	 \frac{1}{2\pi i } \int_{d-i  R}^{d+i  R} \frac{\xi_\mathfrak{C}(s) T^{s+k}}{s(s+1)\cdots (s+k)} d{s} = \frac{c T^{k+1}}{(k+1)!} + \frac{1}{2\pi i } \int_\Gamma \frac{\xi_\mathfrak{C}(s) T^{s+k}}{s(s+1)\cdots (s+k)} d{s}.
	\end{equation*}

We consider the contribution made by each of the line segments in \(\Gamma\).  First, integrating over the interval \([d+i  R, c(R) +i  R]\), we have
	\begin{align*}
	 \left| \int_{d+i  R}^{c(R)+i  R} \frac{\xi_\mathfrak{C}(s) T^{s+k}}{s(s+1)\cdots (s+k)} d{s} \right| 
	 &\le \frac{T^{d+k}}{R^{k+1}} \left| \int_{d+i  R}^{c(R)+i  R} \xi_\mathfrak{C}(s) d{s} \right| 
	 \\
	 &= O\left(\frac{T^{d+k}}{(\log T)^{(k-\rho-1)\varepsilon}} \right)
	\end{align*}
and similarly, for \([c(R) -i  R, d-i  R ]\), we have
	\begin{equation*}
	 \left| \int_{c(R) -i  R}^{d-i  R} \frac{\xi_\mathfrak{C}(s) T^{s+k}}{s(s+1)\cdots (s+k)} d{s} \right| = O\left(\frac{T^{d+k}}{(\log T)^{(k-\rho-1)\varepsilon}} \right).
	\end{equation*}

We estimate the modulus of the integral along \([c(R)+i  R, c(R) - i  R]\) by
	\begin{align*}
	T^{c(R)+k} \left| \int_{c(R)+i  R}^{c(R)-i  R} \frac{\xi_\mathfrak{C}(s)}{s(s+1)\cdots (s+k)} d{s} \right| &= O\left( T^{c(R)+k} \int_1^R t^{1+\rho-k} d{t} \right) \\
	&= O(T^{c(R)+k} R^{2+\rho-k}),
	\end{align*}
which means for any positive \(\gamma\) we have
	\begin{equation*}
	 T^{c(R)+k} R^{2+\rho-k} =
	T^{k+1} e^{-\frac{\log T}{2h^{\rho+1}(\log T)^{\varepsilon\rho}}}(\log T)^{(2+\rho-k)\varepsilon} = O(T^{k+1}(\log T)^{-\gamma}).
	\end{equation*}

Together the integral estimates give us an error term for \(\psi_k(T)\):
	\begin{equation*}
	 \psi_k(T) = c' T^{k+1} + O \left( \frac{T^{k+1}}{(\log T)^{(k-\rho-1)\varepsilon}} \right)
	\end{equation*}
where \(c'>0\). Then repeatedly applying the inequality
	\begin{equation*}
	 \psi_{j-1}(T- \Delta T)\Delta T \le \psi_{j}(T) - \psi_j(T- \Delta T) \le \psi_{j-1}(T) \Delta T,
	\end{equation*}
where 
\[
\Delta T = T(\log T)^{-(k-\rho-1)2^{j-k-1}\varepsilon},
\]
 we obtain 
 \[
 \psi_0(T) = CT + O \left( T(\log T)^{-\delta} \right),
 \]
  where \(C,\delta>0\). Thus we have the
error term  $N_\mathfrak{C}(T) = Ce^{hT/2} + O(e^{hT/2} T^{-\delta})$,
completing the proof of Theorem \ref{thm: polyerror}.

\end{document}